\documentclass[10pt]{amsart}

\usepackage[utf8x]{inputenc}
\usepackage[all]{xy}
\usepackage{amssymb}
\usepackage{amsmath}
\usepackage{hyperref}
\usepackage[all]{xy}
\usepackage{mathrsfs}

\numberwithin{equation}{section}
\newtheorem{thm}{Theorem}[section]
\newtheorem{lem}[thm]{Lemma}
\newtheorem{cor}[thm]{Corollary}
\newtheorem{prop}[thm]{Proposition}
\theoremstyle{definition}
\newtheorem{defi}[thm]{Definition}
\newtheorem{nota}[thm]{Notation}
\newtheorem{conj}{Conjecture}
\theoremstyle{remark}
\newtheorem{rmq}[thm]{Remark}
\newtheorem{exam}[thm]{Example}

\def\ra{\longrightarrow}
\def\N{\mathbb N}
\def\Z{\mathbb Z}
\def\F{\mathbb F}
\def\Fp{{\mathbb F}_p}
\def\F2{{\mathbb F}_2}
\def\U{\mathscr U}
\def\K{\mathcal K}
\def\Ker{{\rm Ker}}

\def\Nil{{\mathcal N}il}
\def\calF{{\mathscr F}}
\def\A{{\mathcal A}_p}
\def\bT{{\mathbb T}}

\def\Hom{\mathrm{Hom}}
\def\Ext{{\mathrm{Ext}}}
\def\Tor{{\mathrm{Tor}}}
\def\ker{\mathrm{ker}}
\def\id{\mathrm{Id}}
\def\Sq{\mathrm{Sq}}
\def\map{\mathrm{map}}
\def\mappt{\mathrm{map}_*}
\def\nlp{\mathrm{nil}}
\begin{document}
\title[Non-realization results]{On non-realization results and conjectures
of N. Kuhn}

\author[Nguyen T.C.]{Nguyen The Cuong}
\address{LIA CNRS Formath Vietnam}
\email{nguyentc@math.univ-paris13.fr}

\author[G. Gaudens]{G\'erald Gaudens}
\address{LAREMA, UMR 6093 CNRS and Université d'Angers}
\email{geraldgaudens@gmail.com}

\author[G. Powell]{Geoffrey Powell}
\address{LAREMA, UMR 6093 CNRS and Université d'Angers}
\email{geoffrey.powell@math.cnrs.fr}

\author[L. Schwartz]{Lionel Schwartz}
\address{LAGA, UMR 7539 CNRS and Universit\'e Paris 13.}
\email{schwartz@math.univ-paris13.fr}

\thanks{This work was partially supported by the VIASM and the program ARCUS
Vietnam
of the \'R\'egion Ile de France and Minist\`ere des Affaires Etrang\`eres'. }

\begin{abstract} We discuss two extensions of results conjectured by  
Nick Kuhn about the non-realization of unstable
algebras as the mod-$p$ singular cohomology of a space, for $p$ a prime. The 
first extends and refines earlier work
of the second and fourth authors, using Lannes' mapping space theorem. The
second (for the prime $2$) 
is based on an analysis of the $-1$ and $-2$  columns of the Eilenberg-Moore 
spectral sequence, and of the associated extension.

In both cases, the statements and proofs use the relationship between the
categories of unstable modules and functors between $\Fp$-vector spaces. 
The second result in particular exhibits the power of the functorial approach. 
\end{abstract}

\maketitle
\section{Introduction}

Let $p$ be a prime number, $\U$ denote the category of unstable modules and  
$\K$ the category of unstable algebras over the mod $p$ Steenrod algebra $\A$
\cite{S94}. The mod-$p$ singular cohomology of a space $X$ is denoted $H^* X$.

In the first part of the paper, the topological spaces $X$ considered are 
  $p$-complete and connected. We assume that $H^* X$ is of finite
type (finite dimensional in each degree) and, moreover, that 
Jean Lannes' functor $T_V$ acts nicely on $H^*X$, in the sense that $T_V H^* X$ 
is  of finite type for all $V$. In order to
apply  Lannes' theory \cite{La92} we also suppose the spaces considered are 
$1$-connected and that  $T_VH^*X$ is $1$-connected for all $V$. For the current 
arguments, the connectivity hypothesis is not a significant restriction,
since it is always possible to 
collapse the $1$-skeleton;  the finiteness hypotheses can be relaxed using
methods of Fabien Morel, as 
explained by Fran\c{c}ois-Xavier Dehon and G\'erald Gaudens \cite{DG03}.

In an earlier paper, the second and fourth authors
gave a proof of the following result, Nick Kuhn's realization conjecture
\cite{K95}:

\begin{thm}
\label{fg}
\cite{GS13} Let $X$ be a space such that $H^*X$ is finitely generated as an
$\A$-module, then $H^*X$ is finite.
\end{thm}

This is a consequence of the following result, Kuhn's strong
realization conjecture \cite{K95}, which uses the Krull filtration
$$
\U _0 \subset \U _1 \subset \U _2  \subset \ldots \subset \U
$$
 of the category $\U$  (see Section \ref{Krull}); in particular,  $\U_0$ is the
full
subcategory  of locally finite unstable modules.

 \begin{thm} 
 \label{KT} 
\cite{GS13} 
Let  $X$ be a space such that  $H^* X \in\U_n$ for some $n \in \N$, 
then $H^*X \in \U_0$
\end{thm}

The aim of this paper is to extend these results in two directions. The first
 exploits the nilpotent filtration (see Section \ref{Nil})
$$
\U = \Nil_0 \supset \Nil_1 \supset \Nil_2 
\supset 
\ldots \supset
\Nil_s \supset \ldots 
$$
of the category of unstable modules.  Here $\Nil_1$ is the full subcategory of
nilpotent unstable modules. For $p=2$, an unstable module 
is nilpotent if the operator $\Sq_0\colon x \mapsto \Sq^{\vert x \vert}x$  acts
nilpotently on any element (a similar definition 
applies for $p$ odd); in particular, a connected unstable
algebra is nilpotent if and only if its augmentation ideal is a 
nilpotent unstable module. 

 Recall that an unstable module is
reduced if it contains no non-trivial nilpotent submodule. The following result 
explains how unstable modules are built from (suspensions 
of) reduced unstable modules: 

\begin{prop}
 \cite{S94,K95}
 An unstable module  $M$ has a natural, convergent decreasing 
filtration
 $\{ \nlp_s M\}_{s\geq 0}$ with $\nlp_s M \in \Nil_s$ and  $\nlp_s M/\nlp_{s+1}
M \cong \Sigma^s R_s M$,  
 where $R_s M$ is a reduced unstable module.
\end{prop}

A reduced unstable module $M$ is said to have degree $n \in \N$  (written  $ 
\deg(M)=n$) if  $M
\in \U_n\setminus \U_{n-1}$, otherwise $\deg (M)= \infty$. Following Kuhn, 
define 
the  profile function 
$w_M\colon \N \rightarrow  \N \cup \{\infty \}$ of an 
unstable module $M$ by 
$$
w_M(i) := \deg (R_i M). 
$$

Recall that the module of indecomposable elements $
QH^*X := \tilde H^*X /(\tilde H^*X)^2,
$ of the cohomology  of a pointed space 
$X$ is an unstable module (here $\tilde H^* X$ is the augmentation ideal). Set 
$w_X:=w_{H^*X}$ and
$q_X :=w_{QH^*X}$.

The following theorem provides a generalization of Theorem \ref{KT} and 
stresses the relationship between the profile functions $w_X$ and $q_X$
 (which is examined in greater detail in Section \ref{sect:profile}).  
A version of the theorem was announced with a sketch proof in \cite{GS12}; 
 the current statement strengthens and unifies existing results.

\begin{thm} 
\label{thm:general}
Let $X$ be a space such that $\tilde H^*X$ is nilpotent.
The following conditions are equivalent:
\begin{enumerate}
\item 
\label{cond:U0}
$H^*X \in \U_0$;
\item 
\label{cond:wX0}
$w_X=0$;
\item 
\label{cond:qX0}
$q_X=0$; 
\item 
\label{cond:wXleqId}
$w_X \leq \id$;
\item 
\label{cond:qXleqId}
$q_X \leq \id$; 
\item 
\label{cond:wX-Id_bounded}
$w_X-\id$ is bounded. 
\end{enumerate}
\end{thm}

To see that this result implies Theorem \ref{KT}, let $X$ be a space such that
$H^* (X) \in \U_n$. 
 This condition implies easily (see Sections \ref{Krull} and \ref{Nil}) that
$w_{\Sigma X} - \id$ is bounded, 
hence $H^* (\Sigma X)$ is locally finite, by Theorem \ref{thm:general}, thus
$H^* X$ also.

Theorem \ref{thm:general} provides evidence  for the following:

\begin{conj}  
\label{conj1}
Let  $X$ be a space such that  $\tilde H^*X$ is
nilpotent.  If $q_X$ is bounded, then $H^*X \in \U_0$.
\end{conj}

This should be compared with the following stronger conjecture, which is
equivalent to the unbounded strong realization conjecture of 
\cite{K95}:

 \begin{conj}  
 \label{conj2}
Let  $X$ be a space such that  $\tilde H^*X$ is
nilpotent.  If $q_X$ takes finite values, then  $H^*X \in \U_0$.
\end{conj}

The second  generalization concerns the first and  second 
layers of the nilpotent filtration. The method of proof is of independent
interest and can be applied in other situations; a generalization of this
approach may lead to a proof of Conjecture \ref{conj1}. 

The fact that the argument is based upon the Eilenberg-Moore spectral 
sequence for computing  $H^*\Omega X$ from $H^* X$ means that the restrictions
upon the space $X$ can be relaxed in the following theorem, in which the prime 
is
taken to be $2$.

\begin{thm} 
\label{thm:R2}  
Let  $X$ be a $1$-connected space such that  $\tilde H^*X$ is of finite type and
nilpotent.  
If  
 $\deg (R_1H^*X)=d \in \N$  then $\deg (R_2H^*X) \geq 2d$.
\end{thm}

The result proved is slightly stronger, giving a precise statement on the cup 
product on $H^*X$ (see Remark 
\ref{rem:strengthen}).

The strategy of proof for Theorem \ref{thm:R2} is different from that of the 
previous results and is related
to that of \cite{S98}. It depends on an 
analysis of the second stage of the Eilenberg-Moore filtration of $H^*\Omega X$
and uses results on triviality and non-triviality of certain extension groups
$\Ext^1_{\calF} (-,-)$,
 where $\calF$ is the category of functors on $\F2$-vector spaces (see Section
\ref{Krull}).

\bigskip

The paper is organized as follows. The Krull filtration is reviewed in Section
\ref{Krull}  and the 
nilpotent filtration in Section \ref{Nil}; using this material, the profile
functions are considered in Section \ref{sect:profile}. Theorem
\ref{thm:general} is proved 
in Section \ref{sect:proofs_GNS_GS}, based on Lannes' theory, which is reviewed 
rapidly. Section \ref{sect:EMSS} 
is devoted to the proof of Theorem \ref{thm:R2}, using the Eilenberg-Moore
spectral sequence.

\begin{rmq}
A first version of this work was made available  by three
of the authors in \cite{NGS14}. The third-named author proposed the current 
approach in Section \ref{sect:EMSS}.
\end{rmq}

\section{The Krull filtration on $\U$}
\label{Krull}

Gabriel introduced the Krull filtration of an abelian category in \cite{Gab62}; 
for
  the category of unstable modules, $\U$, this gives the  filtration 
$$
\U _0 \subset \U _1 \subset \U _2  \subset \ldots \subset \U
$$
by thick subcategories stable under colimits, which  is described in \cite{S94}.

For an abelian category $\mathcal{C}$, the category $\mathcal{C}_0$ is the 
largest thick sub-category generated by the  simple
objects and
stable under colimits; $\U_0$ identifies with the subcategory of locally finite
modules 
($M \in \U$  is locally finite if  $\A x$ is finite for any  $x \in M$).
  The categories $\U_n$ are then defined recursively as follows. 
Having defined $\U_n$, form the quotient
category $\U/\U_n$; then  $\U_{n+1}$ is the pre-image under the canonical 
projection 
$\U\rightarrow \U/ \U_n$ of the subcategory   $(\U/\U_n)_0 \subset (\U/ \U_n)$.

For $M$ an unstable module and $n \in \N$, write  $k_n M$ for the  largest 
sub-module of $M$ that is in $\U_n$.

\begin{prop} 
\cite{K14}
For $M \in \U$, $M = \cup_n k_n M. $
\end{prop}

\begin{exam}
For $k, n \in \N$, $\Sigma^kF(n) \in \U_n \setminus\U_{n-1}$, where  $ F(n) $ 
is the free unstable module  on a generator  of degree $n$. (In the terminology 
of the Introduction, 
$\deg F(n) = n$.)
\end{exam}

\begin{prop}
(\cite[6.1.4]{S94} and \cite{K95}.)
If $M$ is a finitely generated unstable module, 
\begin{enumerate}
 \item 
 there exists $d \in \N$ such that $M 
\in \U_d$;
\item
for each $n \in \N$, $R_n M$ is finitely generated; 
\item 
the nilpotent filtration of $M$ is finite ($\nlp_s M = 0$ for $s \gg 0$).
\end{enumerate}
\end{prop}

There is a characterization of the Krull filtration in terms of Lannes' 
$T$-functor. The functor $T_V$ (for $V$ an elementary abelian $p$-group) is 
left 
adjoint to 
$M \mapsto H^*BV \otimes M$; $T_{\Fp}$ is denoted simply $T$. Since $H^*B\Z/p$ 
splits in $\U$ as  
$\Fp \oplus \tilde H^*B\Z/p$, the functor $T$ is naturally equivalent to $Id 
\oplus \bar T$, where 
$\bar T$ is left adjoint to $\tilde H^*B\Z/p \otimes -$.

\begin{thm}
\cite{La92,S94}
The functor $T_V$ is exact and commutes with colimits; moreover there is a 
canonical isomorphism
$$
T_V(M_1 \otimes M_2) \cong T_V(M_1) \otimes T_V(M_2);
$$
in particular, $
T_V(\Sigma M) \cong \Sigma T_V(M)
$.
\end{thm}

\begin{thm}
\label{krull}
\cite{S94,K14}
The following  are equivalent:
 \begin{enumerate}
  \item
  $M \in \U_n$ ,
 \item 
 $\bar T^{n+1} M=0$.
 \end{enumerate}
\end{thm}

\begin{cor}
\label{krull-t}
If $M \in \U_m$ and $N \in \U_n$ then $ M \otimes N \in \U_{m+n}$.
\end{cor}

There is also a combinatorial characterization of modules in $\U_n$ which are 
of finite type. 
This is stated here for  $p=2$; there are analogous results for odd primes.
Denote by  $\alpha(k)$ the sum of the digits in the binary 
expansion of $k$.
 
\begin{thm}
\cite{S94,S06} 
\label{thm:combinat_krull}
For $M$ an unstable ${\mathcal A}_2$-module and $n \in \N$, 
\begin{enumerate}
 \item 
if $M$ is reduced, $M \in \U_n$ if and only if  $M^j=0$ for $\alpha (j)>n$;
\item 
if $M$ is finitely generated, $M \in \U_n$ if and only if its Poincar\'e
series $\Sigma_i  a_it^i$ has the following property:  there exits 
$k \in \N$ such that, if   $a_d \neq 0$, then  $\alpha 
(d-i) \leq n$, for some  $0 \leq i \leq k$.
\end{enumerate}
\end{thm}

Let $\calF$ be the category of functors from finite dimensional $\Fp$-vector 
spaces to $\Fp$-vector spaces. There is an exact functor \cite{HLS93}  $f 
\colon  \U \ra \calF$ defined by 
$f(M)(V):=\Hom_{\U}(M,H^*(BV))'=T_V(M)^0$ which induces an embedding of 
$\U/\Nil_1 $ in $\calF$.

The functor $\bar T$ corresponds to the difference functor $\Delta : \calF 
\rightarrow \calF$ which is defined on $
F \in \calF$
by
$$
\Delta(F)(V):=\Ker\big(F(V \oplus \Fp) \ra F(V)\big).
$$
Namely, for $M \in \U$, $\Delta (fM) \cong f (\bar T M)$.

Let $\calF_n  \subset \calF $ be the subcategory of polynomial functors of 
degree at most $n$, defined as the full subcategory of functors $F$ such that 
 $\Delta^{n+1}(F)=0$. The polynomial degree of a functor $F$ is written $\deg F 
\in \N \cup \{\infty \}$.

By Theorem \ref{krull}, the following holds:

\begin{prop}
For $n \in \N$, the functor $f : \U \rightarrow \calF$  restricts to an exact 
functor $f : \U_n \rightarrow \calF_n$. 
\end{prop}

\begin{cor}
  \cite{S94} \label{poids}
For $M$ a reduced module, the following are equivalent 
\begin{enumerate}
 \item 
$M\in  \U_d$;
\item 
$deg (fM) \leq d$.
\end{enumerate}
 \end{cor}

\section{The nilpotent filtration}
\label{Nil}

The main results of the paper concern spaces $X$ such that the positive degree 
elements of the cohomology $ H^*X$
are nilpotent; by the restriction axiom for unstable algebras this corresponds 
to $\tilde H^* X$ being nilpotent as an unstable module. 

For $p=2$, the following definition applies, where $\Sq_0$ is the operator $x 
\mapsto \Sq^{|x|}(x)$.  (A similar characterization exists for $p$ odd.)

\begin{defi} 
An unstable $\mathcal{A}_2$-module $M$ is nilpotent if, for any $x \in M$, 
there exists $k$
such that $\Sq_0^kx=0$.
\end{defi}

The archetypal example of a nilpotent unstable module is a suspension and, in 
general, one has:

\begin{prop}
\cite{S94} 
An unstable module is nilpotent if and only if it is the 
colimit of unstable modules  which have  a finite filtration whose quotients 
are suspensions.
\end{prop}

The full subcategory of nilpotent unstable modules is denoted $\Nil_1 \subset 
\U$ and an unstable module is said to be reduced if it contains no 
non-trivial subobject which lies in $\Nil_1$ (this is equivalent to containing
no non-trivial suspension). 

More generally  the category   $\U$ is filtered by  
 thick subcategories $\Nil_s$, $s \geq 0$, where $\Nil_s$ is the smallest thick 
subcategory
stable  under colimits and containing all $s$-fold suspensions:
$$
\U = \Nil_0 \supset \Nil_1 \supset \Nil_2 
\supset 
\ldots \supset
\Nil_s \supset \ldots ~.
$$
 
\begin{prop} 
\cite{S94,K95,K14}
The inclusion $\Nil_s \hookrightarrow \U$ admits a right adjoint 
$\nlp_s : \U \rightarrow \Nil_s$ 
so that    $M \in \U $ has a convergent decreasing filtration
$$
\ldots \subset \nlp_{s+1} M \subset \nlp_s M \subset \ldots \subset M
$$
and  $\nlp_s M_s/\nlp_{s+1} M  \cong \Sigma^s R_sM$,   where $R_s M$ is a 
reduced unstable module.
\end{prop}

The convergence statement is a consequence of the fact that, for any unstable 
module $M$,  $\nlp_s M $ is $(s-1)$-connected.

\begin{prop}
\cite{S94,K95,K14}
\label{nilp}
\ 
\begin{enumerate}
\item 
The $T$-functor restricts to $T \colon \Nil_s \rightarrow \Nil_s$ and $T 
\circ \nlp_s \cong \nlp_s \circ T$.
\item 
The tensor product restricts to  $ \otimes \colon \Nil_s \otimes \Nil_t 
\rightarrow 
\Nil_{s+t}$.
\label{nil}
 \item 
 For $M$  a finitely generated unstable module, the nilpotent 
filtration is finite ({\it i.e.} $\nlp_s M=0$ for $s\gg 0$) and each  $R_s M$ 
is finitely generated.
\item 
An unstable module $M$ lies in $\U_n$ if and only if $R_s M \in \U_n$ for all 
$s \in \N$.
\end{enumerate}
\end{prop}

\begin{proof}
 The result follows from the commutation of $T$ with suspension and the 
respective definitions.
For the final part, since $T$ does not commute with projective limits in 
general, in addition one uses the fact that the category $\U$ is locally 
noetherian \cite[Section 1.8]{S94} and the connectivity of objects of $\Nil_s$.
\end{proof}

\begin{nota}
 For $\varepsilon : K \ra \Fp$ an augmented unstable algebra, denote by $\bar K$ 
the augmentation ideal  $\ker \varepsilon 
$ and 
 $QK\in \U$ the module of indecomposables:
 $
 QK:= \bar K / (\bar K)^2.
 $
\end{nota}

\begin{prop}
\cite[Section 6.4]{S94}
For $K$ an unstable algebra, $QK \in \Nil_1$. If $p=2$, $QK$ is a suspension. 
\end{prop}

Recall (see \cite[Section 6.1]{S94}) that, for $s \in \N$, $\Omega^s : \U \ra 
\U$ is the left adjoint 
to the iterated suspension functor $\Sigma^s$ (hence is right exact) and 
restricts to a functor 
\[
 \Omega^s : \Nil_{k+s} \rightarrow \Nil_k 
\]
for $k \in \N$. The following  is applied in Section \ref{sect:EMSS}.

\begin{lem}
\label{lem:Omega_nil}
 For $M \in \Nil_s$  ($s \in \N$), the natural surjection $M \twoheadrightarrow 
\Sigma^s R_s M$ induces an isomorphism
 \[
  f (\Omega^s M) \stackrel{\cong}{\rightarrow} f (R_s M). 
 \]
If $N \subset M$ is a submodule such that $N \in \Nil_{s+1}$, then the 
surjection $N \twoheadrightarrow M/N$ induces an 
isomorphism $f (\Omega^s M) \stackrel{\cong}{\rightarrow} f(\Omega^s (M/N))$; 
in 
particular $f (\Omega^s (M/N)) \cong f(R_s M)$.
\end{lem}

\begin{proof}
 By hypothesis, there is a short exact sequence of unstable modules:
 \[
  0 
  \ra
  \nlp_{s+1} M 
  \ra 
  M 
  \ra
  \Sigma^s R_s M 
  \ra 
  0
 \]
so that applying $\Omega^s$ gives the exact sequence:
\[
  \Omega^s \nlp_{s+1} M 
  \ra 
  \Omega ^s M 
  \ra
  R_s M 
  \ra 
  0,
 \]
where $ \Omega^s \nlp_{s+1} M \in \Nil_1$ (see \cite[Section 6.1]{S94}). 
Applying the exact functor $f$ gives the isomorphism $f (\Omega^s M)\cong f 
(R_s 
M)$.

The proof of the second statement is similar. 
\end{proof}

The following technical result is used  in the proof of Theorem \ref{thm:R2} in  
Section \ref{sect:EMSS}; for simplicity only 
the case $p=2$ is considered. 

\begin{prop}
\label{prop:restriction}
 Let $ \psi \colon M \rightarrow N$ be a morphism of unstable 
$\mathcal{A}_2$-modules such that 
 \begin{enumerate}
  \item 
  $M \in \Nil_d$ and $N \in \Nil_{d+1}$, for some $d \in \N$;
  \item 
  $R_d M $ is finitely generated.
 \end{enumerate}
Then there exists a finitely generated submodule $U \subset M$ such that
\begin{enumerate}
 \item 
 the restriction $\psi |_U$ is trivial;
 \item 
 the monomorphism $R_d \psi : R_d U \rightarrow R_d M$ has nilpotent cokernel 
(equivalently $f R_d \psi$ is an isomorphism). 
\end{enumerate}
\end{prop}

\begin{proof}
 It is straightforward to reduce to the case where $M$ is finitely generated. 
Moreover, without loss of generality, we may assume that $\psi$ is 
surjective.
 
 For $k \gg 0$ (an explicit bound can be supplied), consider the following 
diagram 
 \[
  \xymatrix{
  & 
  M 
  \ar@{->>}[r]^\psi
  \ar@{->>}[d]
  &
  N
  \\
  \Sigma^d \Phi^k R_d M 
  \ar@{-->}[ur]^{\sigma_k}
  \ar@{^(->}[r]
  &
  \Sigma^d R_d M 
  }
 \]
where $\Phi$ denotes the Frobenius functor (see \cite[Section 1.7]{S94}) and 
the 
bottom arrow is the 
canonical inclusion (recall that $R_dM$ is reduced). The dashed arrow 
denotes a choice of linear section $\sigma_k$ (not $\mathcal{A}_2$-linear in 
general).  

By hypothesis, $M$ is finitely generated, hence so is $N$; thus there 
exists $ h \in \N$ such that $R_s N=0$ if  $ s\not \in [d+1, h]$. Moreover, 
since each $R_sN$ is finitely generated, Theorem \ref{thm:combinat_krull} 
implies 
that $N$ is concentrated in degrees of the form $\ell + t$, $\ell \in [d+1, h]$ 
and $t \in \N$ such that $\alpha (t) \leq D$ for some $D \in \N$. 

Elements of   $\Sigma^d \Phi^k R_d M$ lie in degrees of the form $d + 2^k v$, 
for  $v \in \N$; hence a non-zero element in the image of $\psi \sigma_k$ 
lies in a degree  of the form:
\[
 d +2^kv = \ell + t,
\]
so that $ t = 2^k v - ( \ell - d)$, where $\alpha (t) \leq D$ and $\ell - d\in 
[1, h-d]$, where the values of $D$ and $h$ are independent of $k$.  If
$k$ is sufficiently large, this leads to a contradiction; thus, for $k \gg 0$, 
$\psi \sigma_k =0$.

Consider the submodule $U$ of $M$ generated by the image under 
$\sigma_k$ of a (finite dimensional) space of generators of the unstable module 
$ \Sigma^d \Phi^k R_d M $, which is finitely generated, since $R_dM$ is.

By construction, $\psi \sigma_k = 0$, hence $U \subset \ker \psi$. The functor 
$R_d$ preserves monomorphisms, hence $R_d U \subset R_d M$; 
moreover, from the 
 construction, it is clear that the cokernel  is nilpotent.
\end{proof}

\section{Kuhn's profile functions}
\label{sect:profile}

The profile function of an unstable module is defined using ideas of Kuhn
\cite{K95}:

\begin{defi}
For $M \in \U$, the profile function  $w_M \colon \N \ra \N \cup \{ \infty\}$ 
is defined by:
$$
w_M(i) : = \deg f(R_iM) = \deg R_i M.
$$
\end{defi}

\begin{rmq}
 For $M$ an unstable module, $w_M =0$ if and only if $M$ is locally finite
(i.e. 
$M \in \U_0$).
\end{rmq}

\begin{nota}
 For functions $f, g : \N \ra \N \cup \{ \infty\}$ write $f \leq g$ if 
$f(i)\leq g(i)$ for all $i \in \N$.
 The inclusion $\N \hookrightarrow \N \cup \{ \infty\}$ is denoted $\id$.
\end{nota}

The following operations on functions $\N \ra \N \cup \{ \infty\}$ are useful:

\begin{defi}
 For functions $f, g : \N \ra \N \cup \{ \infty\}$, define functions:
 \begin{enumerate}
  \item 
$f\bullet g (i) := \sup_k \{ f(k) + g (i-k)\}$; 
\item 
$f \circ g (i):=  \sup_{0<k<i} \{ f(k) + g (i-k)\}$ for $i>1$ and $:=0$ 
otherwise; 
  \item 
$\sup \{f,g\} (i) := \sup \{f(i), g(i) \}$ (likewise for 
arbitrary sets of functions);
  \item
$\partial f(i) := \sup \{0, f(i) -1  \}$.
\item
$[f](i):= \sup \{ f(j) | 0 \leq i \leq j \}$.
 \end{enumerate}
\end{defi}

The following states some evident properties:

\begin{lem} 
\label{lem:elem_functions}
For functions $f_i, g_i: \N \ra \N \cup \{ \infty\}$, $i\in \{1,2\}$ such that 
$f_i \leq g_i$:
 \begin{enumerate}
  \item 
  $[f_i ] \leq [g_i]$;
  \item 
  $\partial f_i \leq \partial g_i$;
  \item 
  $f_i \circ g_i \leq f_i \bullet g_i $;
  \item 
  $\sup \{f_1,g_1 \} \leq \sup \{f_2,g_2 \}$;
  \item 
  $f_1 \circ g_1 \leq f_2 \circ g_2$ and $f_1 \bullet g_1 \leq f_2 \bullet 
g_2$. 
 \end{enumerate}
\end{lem}

\begin{prop}
\label{profile_properties}
For $M, N$ unstable modules, 
 \begin{enumerate}
\item 
 \label{convol}
 $w _{M \otimes N} \leq w_M \bullet w_N$ and, if $M, N$ are both nilpotent, $ w 
_{M \otimes N} \leq w_M \circ w_N$; 
\item 
$w_{\bar T M} = \partial w_M$.
\end{enumerate}
For a short exact sequence of unstable modules, $0 \rightarrow M_1 \rightarrow 
M_2 \rightarrow M_3 \rightarrow 0$, the following hold 
  \begin{enumerate}
  \item 
  $w_{M_1} \leq w_{M_2}$; 
  \item 
 $w_{M_2} \leq \sup \{ w_{M_1}, w_{M_3} \}$;
  \item 
$w_{M_3} \leq [w_{M_2}]$.  
  \end{enumerate}
\end{prop}

\begin{proof}
 For tensor products, the  statement holds by the behaviour of the functor 
$R_i$ with respect to tensor products \cite[Proposition 2.5]{K95}. When $M$ and 
$N$ are both nilpotent, $R_0M= 0 = R_0 N$, so that 
the terms $R_0 M \otimes R_k N $ and $R_k M \otimes R_0N$ do not contribute.

The statement for the reduced $T$-functor is a consequence of the compatibility 
of 
$T$ with the nilpotent filtration and the definition of polynomial degree (see
Proposition \ref{nilp} and \cite{S94,K95}).
 
For the short exact sequence,  
the first two properties follow from the left exactness of the  composite 
functor $f 
\circ R_i$ \cite[Corollary 3.2]{K07} and the fact that $\calF_n$ 
is thick. 

For the final point, it suffices to show that $\Sigma^t R_t M_3$ lies in $\U_d$,
where
 $$d= [w_{M_2}](t) = \sup \{ w_{M_2}(i) |0 \leq i \leq t \}.$$
 Now $\Sigma ^t R_t M_3$ is a subquotient 
of $M_2/ \nlp_{t+1} M_2$ and the latter  lies in $\U_d$, since each $\Sigma^i 
R_i M_2$ does, for $0 \leq i \leq t$, by definition of $d$.  
\end{proof}

\begin{rmq}
 The final statement on $w_{M_3}$ can be strengthened slightly; the current 
presentation is sufficient for current purposes.
\end{rmq}

\begin{exam}
Let $M, N$ be unstable modules such that  $w_M \leq \id$ and $w_N \leq \id$, 
then $w_{M \otimes N} \leq \id$. 
Hence, if ${\mathbb{T}(M)}:= \bigoplus_i M^{\otimes i} $ denotes  the tensor 
algebra on $M$, 
$w_{\mathbb{T}(M)} \leq \id$. 
\end{exam}

For $K$ an unstable algebra, Proposition \ref{profile_properties} immediately 
provides a  bound for $w_{QK}$ in terms of $w_K$; the following provides a 
converse.

\begin{prop}
\label{NS} 
For $K$ a connected unstable algebra such that  $\bar K \in \Nil_1$, 
\[
 w_K \leq \sup \{ [w_{QK}] ^{\circ t} | t \in \N \}.
\]
\end{prop}

\begin{proof}
Since $K \cong \bar K \oplus \Fp$ as unstable modules, $w_K = w_{\bar K}$, 
hence 
it suffices to consider the latter.

 The result follows from an analysis of the short exact sequence 
 \[
  0 
  \rightarrow 
  (\bar K)^2 
  \rightarrow 
  \bar K 
  \rightarrow 
  QK 
  \rightarrow 
  0
 \]
together with the surjection $\bar K \otimes \bar K \twoheadrightarrow (\bar 
K)^2$. Proposition \ref{profile_properties} provides the  inequalities
$w_{(\bar 
K)^2} \leq [w_{\bar K \otimes \bar K} ] \leq [w_{\bar K}] \circ [w_{\bar K}
]$, so that
$
 w_{\bar K} \leq \sup \{ w_{QK}, [w_{\bar K}] \circ [w_{\bar K} ] \},
$
 since $\bar K$ is nilpotent. It follows from Lemma \ref{lem:elem_functions} 
that
\begin{eqnarray}
\label{eqn:inequality_w}
[w_{\bar K}] \leq  \sup \{ [w_{QK}], [w_{\bar K}] \circ [w_{\bar K} ]
\}.
\end{eqnarray}

The result follows by induction on $i$ showing that 
\[
 [w_{\bar K}] (i) \leq \sup_t \{ [w_{QK}]^{\circ t} (i) \}.
\]
 The cases $i\in \{0, 1 \}$ are clear and the inductive step uses 
(\ref{eqn:inequality_w}).
\end{proof}

The following is
clear:

\begin{cor}
\label{cor:wK0}
For  $K$ a connected  unstable algebra such that  $\bar K \in \Nil_1$,
$\bar K$ is locally finite if and only if $QK$ is locally finite; equivalently
$w_K=0$ if and only if $w_{QK}=0$.
\end{cor}

\begin{cor}
\label{cor:wK_qK_Id} 
For $K$ a connected unstable algebra such that  $\bar K \in \Nil_1$,  
$w_K \leq \id$ if and only if $w_{QK} \leq \id$.
\end{cor}

\begin{proof}
 The hypothesis $w_K \leq \id$ implies that $[w_K] \leq [\id] = \id$. 
The result follows from the observation that $\ \id \bullet \id = 
\id$. 
\end{proof}

The profile function of an unstable  module $M$ gives information about the 
connectivity of the iterates $\bar T ^n M$ of the 
reduced $T$-functor applied to $M$. The following results are used in the proof 
of Theorem \ref{thm:general} in Section \ref{sect:proofs_GNS_GS}.

\begin{lem}
 \label{lem:profile_T_connectivity}
 For $M$ an unstable module such that $w_M \leq \sup \{ \id + d, 0 \}$ ($d \in 
\Z$) and $0<n \in \N$, $\bar T^n M $ is $(n-d -1)$-connected.
\end{lem}

\begin{proof}
 By definition of the profile function and the hypothesis, 
 $R_i M \in \U_{\sup \{\id +d, 0 \}}$. Hence, by Theorem \ref{krull}, $\bar T^n 
R_i M =0$ for $n > i +d$, with the first non-trivial value for $i=n-d$ (here it 
is essential that $n>0$). 
 By compatibility of the action of the $T$-functor with the nilpotent 
filtration 
(see Proposition \ref{nilp}), it follows that 
 \[
 \bar T ^n M = \nlp_{n-d}    \bar T ^n M.
  \]
This implies, in particular, that $\bar T^n M$ is $(n-d-1)$-connected.
\end{proof}

This result can be applied in the study of unstable modules $M$ for which $w_M 
- 
\id$ is bounded:

\begin{lem}
 \label{lem:w_M-id-bounded}
 Let $M \in \U$ be an unstable module such that $w_M - \id$ is bounded and set 
 \begin{eqnarray*}
  d&:=& \sup \{ w_M(i) -i \} 
  \\
  s&:=& \inf \{ i | w_M (i) -i = d \}.
 \end{eqnarray*}
 If $d>0$, then $\bar T^{s+d-1} M \in \Nil_s$  (equivalently $R_i \bar T^{s+d-1} 
M =0 $ for 
$i<s$) and 
\[
 R_i  \bar T^{s+d-1} M 
 \in 
 \left\{ 
 \begin{array}{ll}
  \U_1 \backslash \U_0 & i=s \\
  \U_{i-s+1} & i>s.
 \end{array}
 \right.
\]
In particular, $w_{\bar T^{s+d-1} M} - \id$ is bounded above by $1-s$ and, for 
$0<n\in \N$, 
$
 \bar T^{n+s+d-1} M \in \Nil_{n+s-1},
$ 
thus $ \bar T^{n+s+d-1} M $ is $(n+s-2)$-connected.
\end{lem}

\begin{proof}
 The first statement follows from the relationship between $\bar T$ and 
$\partial$ given in Proposition \ref{profile_properties},
  in addition checking the vanishing of  $R_i \bar T^{s+d-1} M$ for $i<s$. The 
second part follows by applying Lemma \ref{lem:profile_T_connectivity}. 
\end{proof}

\section{Proof of Theorem \ref{thm:general}}
\label{sect:proofs_GNS_GS}

We commence by a rapid review of Lannes' theory, which is the main ingredient 
in the proof of Theorem \ref{thm:general} and is also the reason for the
restrictions 
imposed on the topological spaces considered. 

Let  $X$ be a $p$-complete, $1$-connected space and  assume that $TH^*X$ is of 
finite type  and $1$-connected.  The evaluation
map
  $$
  B\Z/p \times \map (B\Z/p, X) \ra X
  $$
 induces a map in cohomology
$
H^*X \ra H^* B\Z/p \otimes H^* \map(B\Z/p,X)
$
and hence, by adjunction, 
$
TH^*X \ra H^* \map (B\Z/p,X)
$.

\begin{thm}  
\cite{La92} \label{Lannes2}
Under the above hypotheses, the natural map
$
T H^*X \to H^* \map (B\Z/p,X)
$
is an isomorphism of unstable algebras.
\end{thm}

\begin{nota}
\cite{K95}
 Denote by $\Delta (X)$ the homotopy cofibre of the map $X \ra
\map (B\Z/p, X)$ induced by $B\Z/p \rightarrow *$.
\end{nota}

Theorem \ref{Lannes2} yields the following:

 \begin{prop}
 \label{prop:w_Delta}
 Under the above hypotheses on $X$, 
 $H^* \Delta(X) \cong \bar T H^*X$, 
 hence 
 \begin{enumerate}
 \item 
 $w_{\Delta(X)}=\partial w_X$;
   \item 
 if $H^*X \in \U_n$, then $H^*\Delta(X)  \in \U_{n-1}.$ 
 \end{enumerate}
\end{prop}

The functor $T$ induces  $T : \K \rightarrow \K$, which commutes with 
the indecomposables functor for augmented unstable algebras \cite[Lemma 
6.4.2]{S94}: namely, for $K$ an augmented unstable algebra, 
there is a natural isomorphism 
$
T(QK) \cong Q(TK).
$ 
There  is no analogous statement  for $\bar T$; however, if  $Z$ is an  
$H$-space, there is a  homotopy equivalence:
$$
\map(B\Z/p,Z) \cong Z \times \mappt(B\Z/p,Z).
$$

For example, this leads to:
 
\begin{prop}
\label{indec}
 \cite{CCS07}
For $Z$ an $H$-space (also satisfying the global hypotheses)
$$Q H^* \mappt
(B\Z/p^{\wedge n}, Z )\cong  \bar{T}^n QH^* Z .
$$
\end{prop}

\begin{nota}
For $X$  a connected space, write $ w_X:=w_{H^*X}$ 
and $q_X:=w_{QH^*X}$.
\end{nota}

\begin{proof}[Proof of Theorem \ref{thm:general}]
Proposition \ref{nilp} implies that an unstable module $M$ is locally finite if 
and
only if $w_M=0$, which implies
the equivalence $(\ref{cond:U0}) \Leftrightarrow (\ref{cond:wX0})$. Corollary
\ref{cor:wK0} gives 
the equivalence  $(\ref{cond:wX0}) \Leftrightarrow (\ref{cond:qX0})$ and
Corollary \ref{cor:wK_qK_Id}  
the equivalence $(\ref{cond:wXleqId}) \Leftrightarrow (\ref{cond:qXleqId})$.
The implications 
$(\ref{cond:wX0}) \Rightarrow (\ref{cond:wXleqId})$ and $(\ref{cond:wXleqId})
\Rightarrow (\ref{cond:wX-Id_bounded})$ are clear,  hence it suffices to
establish: 
\begin{itemize}
\item 
$
 (\ref{cond:wX-Id_bounded})
\Rightarrow (\ref{cond:wX0}) $: if $w_X - \id$ is bounded then $w_X =0$. 
\end{itemize}

Suppose that there exists a space  $X$ (satisfying the global hypotheses) such
that $0 \neq w_X \leq \id$ and $\tilde H^* X \in \Nil_1$; {\em reductio ad 
absurdam}. 

The first step is analogous to Kuhn's reduction  \cite{K95}. As in Lemma 
\ref{lem:w_M-id-bounded}, 
set $d:= \sup \{ w_X (i) -i \}$ and  $s := \inf \{ i | w_X (i) -i = d \}$. 
Proceeding 
as in Lemma \ref{lem:w_M-id-bounded}, set $Y:= (\Delta^{s+ d -1} X)^{\wedge}_p 
$, so that 
Proposition \ref{prop:w_Delta} gives
\begin{eqnarray*}
w_Y (i)  &=&0, \  i < s \\
w_Y (s) &=& 1 \\
w_Y (i) &\leq & i-s +1, \ i > s.
\end{eqnarray*}
Moreover, by collapsing down a low-dimensional skeleton and $p$-completing, one 
can arrange that 
 $R_iH^*Y=0$  for  $0<i<s$.

In order to work with pointed mapping spaces, consider  $Z := \Omega (\Sigma 
X)^{\wedge}_p$. By the Bott-Samelson theorem,    
 $H^*Z \cong \mathbb{T}(\tilde H^* X)$ as an unstable module, hence 
the global hypotheses are satisfied by $Z$. Moreover, by 
 Proposition \ref{profile_properties}, one has 
 $R_i H^* Z =0$ for $0<i <s$ and  $R_i  H^* Z \in \U_{i-s+1}$ for $i 
\geq s$. Lemma \ref{lem:profile_T_connectivity} therefore implies that 
 $\bar T ^n H^*Z$  is $(n+s-2)$-connected for $n >0$ (and $H^* Z$ is 
$s-1$-connected).   

 It follows from \cite{BK87} (in particular using Chapters 6 and 9 to show that 
there are no phantom maps) that 
 $\mappt(B\Z/p^{\wedge n},Z)$ is $(n+s-2)$-connected.

By construction, $R_s H^* Z$ is a reduced unstable module in $\U_1$, hence (by 
\cite[Proposition 0.6]{K95}, for example) there exists a non-trivial morphism 
$R_s H^* Z \rightarrow F(1)$, where $F(1)$ is the free unstable module on a 
generator of degree $1$. Composing with the canonical inclusion $F(1) 
\hookrightarrow \tilde H^* B \Z/p$ gives 
$$
\varphi^*_s: H^* Z \twoheadrightarrow  H^*Z/\nlp_{s+1}H^* Z \cong \F2 \oplus 
\Sigma ^s R_s H^* Z \to
\F2 
\oplus \Sigma ^s F(1)  \subset \F2 \oplus \Sigma^s  \tilde H^* B\Z/p,
$$
that is a morphism of unstable algebras, by compatibility of the nilpotent 
filtration with  multiplicative structures.

By Lannes' theory \cite{La92}, the morphism $\varphi^*_s$ can be realized as the
cohomology of a map 
$\varphi_s : \Sigma^s B \Z/p \rightarrow Z$, by applying  Theorem \ref{Lannes2} 
together with the Hurewicz theorem, since  $\mappt(B\Z/p,Z)$ is
$(s-1)$-connected. 

Thus, consider the extension problem
\[
 \xymatrix{
 \Sigma^s B \Z/p \ar[r]^{\varphi_s} 
 \ar[d]
 &
 Z
 \\
 \Sigma^{s-1} K (\Z/p, 2),
 \ar@{.>}[ur]
}
 \]
where the vertical map is the $(s-1)$-fold suspension of the canonical map
$\Sigma B \Z/p 
\rightarrow K (\Z/p , 2)$. 

Algebraically it is clear that no such extension can exist, since $\varphi_s^*$ 
is non-trivial in positive degrees (by construction), whereas 
$$
\Hom_\U (\tilde H^* Z, \tilde H^*\Sigma^{s-1}K(\Z/p, 2)) =0
$$
since $\tilde H^* Z \in \Nil_s$ and $\tilde H^* \Sigma^{s-1}K(\Z/p, 2)$ is the 
$(s-1)$-fold suspension of a reduced module ($H^*K(\Z/p,2)$ is reduced by  
Proposition \ref{EML} below).

However, obstruction theory shows that one can construct such a factorization, 
as follows. 
Recall that $K(\Z/p,2) \simeq B (B \Z/p) $ can be built, starting from  $\Sigma 
B\Z/p$, using 
Milnor's 
construction: there is a filtration   $* =C_0 \subset C_1= \Sigma 
B\Z/p 
\subset 
C_2 \subset \ldots 
\subset \cup_n C_n=K(\Z/p,2)$
 with associated (homotopy) cofibre sequences 
$$
\Sigma^{n-1}B\Z/p^{\wedge n}  \rightarrow  C_{n-1} \rightarrow C_{n} .
$$

The associated obstructions to extending  $\varphi_s \colon \Sigma^s 
B\Z/p \rightarrow Z$  lie in the pointed homotopy groups 
$$
[\Sigma ^{n+s-2}  (B\Z/p)^{\wedge n}, Z]= \pi_{n+s-2}
\mappt( B\Z/p^{\wedge n},  Z).
$$ 
The groups $\pi_{n+s-2}\mappt (B\Z/p^{\wedge n},  Z)$ are trivial, since 
$\mappt(B\Z/p^{\wedge n},Z)$ is $(n+s-2)$-connected, as observed above. 
It follows that an extension exists, which is a contradiction to the existence
of such a space $Z$,  completing the proof. 
\end{proof}

\begin{prop}
\label{EML}
The unstable module $H^*K(\Z/p,2)$ is reduced.
\end{prop}

\begin{proof}
This result is well known to the experts, and holds for $H^*K(\Z/p,n)$ 
for any $n \in \N$. For $p=2$ the result is established in \cite{K98}; a proof 
is sketched here for $p$ odd, since we do not know a convenient reference.

The cohomology $H^*K(\Z/p,2)$ is isomorphic (with the usual notation) to
 $$\Fp[x, \beta P^{I_h}\beta (x)] \otimes E(P^{I_\ell} \beta (x)) \cong  \Fp[x, 
x_h | h\geq 1]
\otimes E(y_\ell |\ell\geq 0) $$
with $\vert x \vert=2$, $I_h=(p^{h-1},p^{h-2},\ldots, p,1)$, $h\geq 1$, and 
$\beta (x) =y_0$; in particular,   $\beta(y_h)=x_h$, $h \geq 1$. It is enough 
to show that, for any non zero 
element $z \in \tilde
H^*K(\Z/p,2)$, there exists an operation $\theta$ such that $\theta (z) \in 
\Fp[x, x_h] $ and $\theta (z)
\not =0$. 

Any element $z$ can be written as a sum   $\sum_{0 \leq i \leq t} \sum_j 
P_{i,j}(x,x_h) \otimes 
L_{i,j}(y_\ell)$.
If $z$ has degree $2n$ and there is   a non-trivial term with exterior part of 
degree $0$,  a straightforward 
application of the Cartan formula shows 
that $\theta = P^n$ suffices, since the reduced powers act trivially on the
exterior 
generators. 

For a general element, one reduces to  such elements by applying
operations which 
decrease (non-trivially) the length of the exterior factors that occur. Consider
amongst the exterior 
factors 
 a term $L_{i,j}$ of minimal length, and the minimal $\ell$ for which  $y_\ell$ 
occurs in it.
 Let $Q_i$ be the usual Milnor derivation, 
$[\beta,Q_i]$ is also a derivation; it acts trivially on the $x_h$, sends $x$ 
to 
$y_{i+1}$, 
and $y_\ell$  to  
$x_{i-\ell+1}^{p^\ell}$. Using a standard lexicographic 
order argument, one can see that this operation does  the job for $i$ large
enough.
\end{proof}


\section{Using the Eilenberg-Moore spectral sequence}
\label{sect:EMSS}

In this section, the prime $p$ is taken to be $2$ and the space $X$ is 
$1$-connected such that  $\tilde H^*X$ is 
nilpotent and of finite type. The objective of this section is to prove the 
following, which is equivalent to Theorem \ref{thm:R2} of the Introduction.

\begin{thm}\label{infini} Let $X$ be a space such that $\tilde H^*X $ is of 
finite type and is nilpotent.
If  $w_X(1) =d  \in \N  $ then  $w_X(2)  \geq 2d  $.
\end{thm}

The interest  of the result is to give some control 
on $R_2H^*X$, starting from information about  $R_1H^*X$.
See Remark \ref{rem:strengthen} for a slightly refined version of this theorem.

\begin{rmq}
The theorem is stated only for $p=2$; the difficulties that occur in the odd 
primary case in  \cite{S98,S10} also arise here, but look more manageable. 
\end{rmq}

The method was originally suggested by the following observation:

\begin{prop} 
Let $M$ be a connected, reduced unstable module such that $\deg (f M ) = d \in 
\N$. If $d>0$, then the unstable module $\bT (M)= \bigoplus_i M^{\otimes i}$ 
does not carry 
the structure of an unstable algebra.
\end{prop}

\begin{rmq}
This result is a special case of a general structure result for reduced
unstable  algebras.  From the viewpoint of this paper, heuristically the idea 
is 
that cup products of 
classes in $M$ should appear in $M \otimes M$ whereas the restriction axiom for 
unstable algebras implies that cup squares occur in $M$. Thus, the triviality
of the extension between $M$ and $M \otimes M$ is incompatible with an unstable 
algebra structure.
\end{rmq}

The proof of Theorem \ref{infini} is based on the analysis of an $\Ext^1_\calF$ 
group, playing off the following  non-splitting result against a 
vanishing criterion. 

Recall from Section \ref{Krull} that $\calF$ denotes the category of functors 
on 
$\F2$-vector spaces; 
a functor of $\calF$ is finite if it has a finite composition series. As usual, 
$S^n$ denotes the $n$th symmetric
power functor and $\Lambda^n$ the $n$th exterior power functor.

\begin{lem}
\label{lem:non-split}
For $F \in \calF$ a non-constant finite functor, post-composition with 
 the short exact sequence $ 0 \rightarrow S^1 \rightarrow S^2 \rightarrow 
\Lambda ^2 \rightarrow 0$, where $S^1 \rightarrow S^2$ is the Frobenius 
$x\mapsto x^2$,  induces an exact sequence
 $$
 \xymatrix{
 0
 \ar[r]&
 F 
 \ar[r]
 &
 S^2 (F) 
\ar[r]
&
\Lambda^2 (F) 
\ar[r]
&
0
}
$$
which does not split.
\end{lem}

\begin{proof}
The result follows from \cite[Theorem 4.8]{KIII}, since the Frobenius fits into 
the non-split short exact sequence  $ 0 \ra 
S^1 \rightarrow S^2 
\rightarrow \Lambda^2 \ra  0$. 
\end{proof}

\begin{nota}
For $F\in \calF$  a finite functor of polynomial degree $d$, set 
$\overline{F}:= 
\ker 
\{ F \twoheadrightarrow q_{d-1} F \}$, where $q_{d-1} F$ is the largest 
quotient of degree $\leq d-1$. 
\end{nota}

Lemma \ref{lem:non-split} will be played off against the vanishing result for
$\Ext^1_\calF$ in the following statement.

\begin{lem} 
\label{nonsplit}
Let $F$ be a finite functor of polynomial degree $d$, $G_{<d}$  of degree $< 
d$, 
$G_{<2d}$ 
 of degree $<2d$ and $G_{\leq d}$ of degree $\leq d$.  Then
\begin{enumerate}
\item 
$\Hom_\calF (\overline{F}, G_{<d}) =0$;
 \item 
$\Hom_\calF (\overline{F} \otimes \overline{F} , G_{<2d}) = 0 = \Hom_\calF
(\Lambda ^2(\overline{F}) , G_{<2d})$;
\item 
 $\Ext_\calF^1 (\overline{F} \otimes \overline{F}, G_{\leq d} ) =0$. 
\end{enumerate}
\end{lem}

\begin{proof}
The first statement is a consequence of the definition of $\overline{F}$.

The second is similar and follows from the compatibility of the polynomial
filtration of $\calF$ with tensor products, which implies that $\overline{F}
\otimes \overline{F}$ has no quotient of polynomial degree $<2d$. The second 
equality follows from 
 the fact that the functor $\Lambda ^2(\overline{F})$ is a quotient of
$\overline{F} \otimes \overline{F}$.

The result for $\Ext^1_\calF$ is proved as follows. 
Using {\em d\'evissage} it is straightforward to reduced to the case where
$G_{\leq d}$ is simple. 
Since a simple functor of polynomial degree $n$ embeds in the $n$th tensor
functor $T^n : V \mapsto V^{\otimes n}$, which 
is finite and has polynomial degree $n$, using the previous statement for
$\Hom_\calF$, it suffices to show that  $\Ext^1_{\calF} (\overline{F} \otimes
\overline{F}, T^n) =0$ for $n \leq d$. This follows by the standard methods
introduced in \cite{FLS94}, exploiting the tensor
product on the left hand side.
\end{proof}

The proof of Theorem \ref{infini} uses these results in conjunction with the 
Eilenberg-Moore spectral sequence. 
For relevant details (and further references) on the Eilenberg-Moore spectral 
sequence computing 
$H^* \Omega X$ from $H^* X$, see \cite[Section 8.7]{S94}. Note  
that the hypothesis 
that $X$ is simply-connected ensures strong convergence of the spectral 
sequence.

Recall that the $(-2)$-layer of the Eilenberg-Moore filtration 
$F_{\infty}^{-2,*} $ on 
$H^*\Omega X$  is an extension in unstable modules between the  column  
$E^{-1,*}_\infty$ desuspended once and the  column  $E^{-2,*}_\infty$  
desuspended twice:
$$
 0 \ra 
 \Sigma^{-1} E_{\infty}^{-1,*}
 \ra
 F_{\infty}^{-2,*}
 \ra 
\Sigma^{-2}E_{\infty}^{-2,*} \ra 0.
$$

The term $E_{\infty}^{-1,*} $ is a quotient of $Q H^*X \cong E_{2}^{-1,*}$ by a 
submodule in $\Nil_2$, by \cite[Theorem 8.7.1]{S94}. Similarly,  
$E_{\infty}^{-2,*} $ is a quotient of $\Tor^{-2}_{H^*X}(\F2,\F2)$ (which
belongs to $\Nil_2$ by \cite[Theorem 6.1]{S94}),  
by a submodule in ${\Nil}_3$  (see \cite[Proposition 8.7.7]{S94}).  
(In the reference $\overline{\Nil}_3$ is used, where $\overline{\Nil}_3$ is 
generated by $\Nil_3$ and $\U_0$ 
\cite[Section 6.2]{S94}, however here the situation reduces to ${\Nil}_3$.)

From Lemma \ref{lem:Omega_nil}, it follows that  applying the exact functor $f 
: 
\U \rightarrow \calF$ yields the short 
exact sequence:
\begin{eqnarray}
\label{eqn:f_ses}
0 
\rightarrow 
f R_1 Q  H^* X 
\rightarrow 
f  F_{\infty}^{-2,*}
\rightarrow 
f R_2 \Tor^{-2}_{H^*X}(\F2,\F2)
\rightarrow 
0.
\end{eqnarray}
Moreover, the surjection $\tilde H^*X \twoheadrightarrow Q H^*X$ induces an 
isomorphism   $ fR_1 \tilde H^*X \cong f R_1 Q H^*X$. This
allows  arguments to be carried out in the category of functors $\calF$.

\begin{nota}
 Write 
 \begin{eqnarray*}
   F_1 &:=& fR_1 \tilde H^*X \cong f(\Sigma^{-1}E^{-1,*}_\infty)
    \\ 
 F_2&:= &f(F^{-2,*}_\infty),
 \end{eqnarray*}
so that $F_2/ F_1 \cong f R_2 
\Tor^{-2}_{H^*X}(\F2,\F2)$ by (\ref{eqn:f_ses}).
 \end{nota}

 \begin{rmq}
  With this notation, the short exact sequence (\ref{eqn:f_ses}) represents a 
class  $[F_2] \in \Ext_\calF^1 (F_2/F_1, F_1)$.
 \end{rmq}

The compatibility of the cup product with the Eilenberg-Moore spectral sequence 
gives a morphism $S^2 (F_1) \rightarrow F_2$ and hence a morphism of
short exact sequences 
$$
 \xymatrix{
 0 
 \ar[r]&
 F_1 
\ar@{=}[d]
\ar[r]
 &
 S^2 (F_1) 
\ar[r]
\ar[d]
&
\Lambda^2 (F_1) 
\ar[r]
\ar[d]
&
0
\\
0
 \ar[r]&
 F_1 
 \ar[r]
 &
\ar[r]
F_2 
\ar[r]
&
F_2/F_1 
\ar[r]
& 0
 }
$$
so that the right hand square is a pull-back.

By hypothesis, the functor $F_1$ is of degree $d$;  the inclusion
$\overline{F_1} 
\subset F_1$ induces inclusions $S^2 
(\overline{F_1})  \subset S^2 ( F_1)$ and  $\Lambda^2 (\overline{F_1})  \subset 
\Lambda ^2 ( F_1)$ which fit into the following diagram of morphisms of short 
exact sequences:
\begin{eqnarray}
 \label{eqn:ses_diagram}
 \xymatrix{ 
 0
\ar[r]&
\overline{ F_1} 
\ar@{^(->}[d]
\ar[r]
 &
S^2 ( \overline{F_1} )
\ar[r]
\ar[d]
&
\Lambda^2 (\overline{F_1}) 
\ar[r]
\ar@{=}[d]
&
0
\\
0
 \ar[r]&
 F_1 
\ar@{=}[d]
\ar[r]
 &
 {E}
\ar[r]
\ar[d]
&
\Lambda^2 (\overline{F_1}) 
\ar[r]
\ar[d]
&
0
\\
0
 \ar[r]&
 F_1 
\ar@{=}[d]
\ar[r]
 &
 S^2 (F_1) 
\ar[r]
\ar[d]
&
\Lambda^2 (F_1) 
\ar[r]
\ar[d]
&
0
\\
0
 \ar[r]&
 F_1 
 \ar[r]
 &
\ar[r]
F_2 
\ar[r]
&
F_2/F_1 
\ar[r]
& 0
 }
\end{eqnarray}
in which  the second row can be viewed either as the pushout of the top row or 
the pullback of the second.  Moreover, the first and  third rows are not split, 
by Lemma \ref{lem:non-split}, hence the bottom row is not 
split. 

\begin{lem}
\label{lem:non-split_bar}
 If $F_1$ is non-constant the short exact sequence 
 $$
  0
  \rightarrow 
  F_1 
  \rightarrow 
  {E}
  \rightarrow 
  \Lambda^2 (\overline{F_1}) 
  \rightarrow 
  0
 $$
does not split. 
\end{lem}

\begin{proof}
Consider the long exact sequence for $\Ext_\calF^*$ induced by the defining 
short exact 
sequence  $\overline{F_1} \rightarrow F_1 \rightarrow q_{d-1} F_1$, which gives 
the exact 
sequence:
 $$
  \Hom (  \Lambda^2 (\overline{F_1}) , q_{d-1} F_1) 
  \rightarrow 
  \Ext_\calF^1 (\Lambda^2 (\overline{F_1}) ,\overline{F_1})
  \rightarrow 
   \Ext_\calF^1 (\Lambda^2 (\overline{F_1}) ,F_1)
 $$
in which the first term is zero, by Lemma \ref{nonsplit}. Thus the second 
morphism is injective.

The top row in diagram (\ref{eqn:ses_diagram}) represents a  non-trivial class 
in  $\Ext^1_\calF
(\Lambda^2 (\overline{F_1}) ,\overline{F_1})$, by Lemma  
\ref{lem:non-split}, hence pushes out to a non-split short exact sequence
 represented by a non-zero class in $  \Ext_\calF^1 (\Lambda^2 (\overline{F_1}) 
,F_1)$.
\end{proof}

\begin{prop}
\label{prop:non-factorization}
 The morphism $\Lambda^2 \overline{F_1} \rightarrow F_2/ F_1$ does not factor 
over $\overline{F_1} \otimes \overline{F_1} $:
 $$
  \xymatrix{
  \Lambda^2 \overline{F_1} 
  \ar[r]
  \ar[d]
  &
  \overline{F_1} \otimes \overline{F_1}
  \ar@{.>}[ld] ^{\not \exists}
  \\
  F_2/F_1. 
  }
 $$
\end{prop}

\begin{proof}
The non-trivial extension of Lemma  \ref{lem:non-split_bar} is the image of 
the class $[F_2] \in \Ext_{\calF}^1 (F_2/F_1, F_1)$ under the morphism 
induced by $\Lambda ^2 (\bar F_1) \rightarrow F_2/ F_1$. Hence there can be no 
factorization across the group  $\Ext_\calF^1 (\overline{F_1} \otimes 
\overline{F_1}, F_1)$, which is trivial  by Lemma \ref{nonsplit}.
\end{proof}

\begin{prop} 
\label{prop:tor}
Assume that $\tilde H^*X$ is nilpotent and of finite type. 

If $deg(f 
R_1H^*X) 
= d \in \N$ and  $deg(f  R_2(H^*X)) <2d$, then
\begin{enumerate}
 \item 
 there is an inclusion $\overline{F_1} \otimes \overline{F_1} \hookrightarrow f 
R_2 \Tor^{-2}_{H^* X} (\F2, \F2) \cong F_2/F_1$ with cokernel of degree $<2d$;
 \item 
 the morphism $\Lambda^2 (\overline{F_1}) \rightarrow F_2/F_1$ factors across 
$\overline{F_1} \otimes \overline{F_1} \hookrightarrow f 
R_2 \Tor^{-2}_{H^* X} (\F2, \F2)$.
 \end{enumerate}
\end{prop}

\begin{proof}
 The cup product of $H^*X$ induces a morphism $\tilde H^*X \otimes \tilde H^* X 
\rightarrow \nlp_2 H^* X$, since $\tilde H^* X \in \Nil_1$, hence in $\calF$
 \[
F_1 \otimes F_1 =   f R_1 H^* X \otimes fR_1 H^* X \rightarrow f R_2 H^* X.
 \]
Restricting to $\overline{F_1}\otimes \overline{F_1} \subset F_1 \otimes F_1$ 
this morphism is trivial, by Lemma \ref{nonsplit}, since $\deg f R_2 H^* X
<2d$, 
by hypothesis.

Lift $\overline{F_1}$ to a submodule $M$ of $\tilde H^* X$ (so that $fR_1 M$ 
corresponds to $\overline{F_1}$) and consider the restriction of the product.
By 
construction this gives:
\[
 M \otimes M \rightarrow \nlp_3 H^* X
\]
with $M \otimes M \in \Nil_2$. Moreover, it is easily checked that the 
finiteness hypothesis required to apply Proposition \ref{prop:restriction} is 
satisfied, hence there exists a finitely generated submodule $U \subset M 
\subset \tilde
H^* 
X$ such that $f R_1 U = f R_1 M = \overline{F_1}$ and the cup product restricts 
to a trivial map $U \otimes U \rightarrow \tilde H^* X$. (This is a slight 
extension of Proposition \ref{prop:restriction} using the 
fact that the choices in the proof can be taken to be compatible with the 
tensor 
product $M \otimes M$.)

The  calculation of $\Tor^{-2}_{H^* X} (\F2, \F2)$ using the bar 
construction then implies that there is an inclusion $\overline{F_1} \otimes
\overline{F_1} \hookrightarrow f 
R_2 \Tor^{-2}_{H^* X} (\F2, \F2)$. Moreover, by the definition of
$\overline{F_1} \subset F_1$, it is clear
that the cokernel has degree $<2d$. 

Finally, again by Lemma \ref{nonsplit}, there is no non-trivial map from 
$\Lambda^2
(\overline{F_1})$ to a functor 
of degree $<2d$, which gives the factorization statement. 
\end{proof}

\begin{rmq}
 The appeal to Proposition \ref{prop:restriction} can be avoided 
so that the argument is carried out entirely within $\calF$.
\end{rmq}

\begin{proof}[Proof of Theorem \ref{infini}]
Suppose that $deg(f  R_2H^*X) <2d $, then  Proposition \ref{prop:tor} 
provides a factorization
 \[
  \Lambda^2 (\overline{F_1}) \rightarrow \overline{F_1}\otimes \overline{F_1} 
\rightarrow F_2/ F_1. 
 \]
 This contradicts Proposition \ref{prop:non-factorization}.
\end{proof}

\begin{rmq}
\label{rem:strengthen}
 The result proved is slightly stronger, without supposing  $w_X (2) 
<2d$. Namely the argument shows that the composite:
\[
 \overline{F_1} \otimes \overline{F_1} \subset F_1 \otimes F_1 \rightarrow f
R_2 
H^*X,
\]
where the second morphism is induced by the cup product of $H^* X$, is 
necessarily non-trivial. 
\end{rmq}


\providecommand{\bysame}{\leavevmode\hbox to3em{\hrulefill}\thinspace}
\providecommand{\MR}{\relax\ifhmode\unskip\space\fi MR }
\providecommand{\MRhref}[2]{%
  \href{http://www.ams.org/mathscinet-getitem?mr=#1}{#2}
}
\providecommand{\href}[2]{#2}

\end{document}